\documentclass[11pt,a4paper]{article}

\usepackage{amssymb}
\usepackage{amsthm}
\usepackage{amsmath}
\usepackage{amsfonts}
\usepackage{cancel}
\usepackage{pgf,tikz}
\usepackage{float}
\usepackage{multirow}
\usepackage{url}

\def\NZQ{\mathbb}
\def\ZZ{{\NZQ Z}}

\theoremstyle{plain}
\newtheorem{theorem}{Theorem}[section]
\newtheorem{lemma}[theorem]{Lemma}
\newtheorem{proposition}[theorem]{Proposition}
\newtheorem{corollary}[theorem]{Corollary}
\newtheorem*{open}{Open question}

\theoremstyle{definition}

\newtheorem{example}{Example}
\theoremstyle{remark}
\newtheorem{remark}[theorem]{Remark}

\def\opn#1#2{\def#1{\operatorname{#2}}} 
\opn\height{height}

\newcommand{\CC}{\mathcal{C}}
\newcommand*\rot{\rotatebox{90}}

\newcommand\blfootnote[1]{%
  \begingroup
  \renewcommand\thefootnote{}\footnote{#1}%
  \addtocounter{footnote}{-1}%
  \endgroup
}

\begin{document}


\title{Cohen--Macaulay generalized binomial edge ideals}

\author{Luca Amata, Marilena Crupi and Giancarlo Rinaldo}

\newcommand{\Addresses}{
{\footnotesize
\begin{center}
\textsc{Department of Mathematics and Computer Sciences,\\
Physics and Earth Sciences, University of Messina,\\
Viale Ferdinando Stagno d'Alcontres 31, 98166 Messina, Italy}

\textit{E-mail addresses}: \texttt{luca.amata.77@gmail.com};  \texttt{mcrupi@unime.it}; \texttt{giancarlo.rinaldo@unime.it}
\end{center}
}}
\date{}
\maketitle
\Addresses

\begin{abstract}
Let $G$ be a simple graph on $n$ vertices and let $J_{G,m}$ be the generalized binomial edge ideal associated to $G$ in the polynomial ring $K[x_{ij}, 1\le i \le m, 1\le j \le n]$.
We classify the Cohen--Macaulay generalized binomial edge ideals. Moreover, we study the unmixedness and classify the bipartite  and power cycle unmixed ones.

\blfootnote{
\hspace{-0,3cm} \emph{Keywords:} Generalized binomial edge ideals, Circulant graphs, Dual graph, Cohen--Macaulayness, Unmixedness.


* \emph{Corresponding author: Giancarlo Rinaldo; email: giancarlo.rinaldo@unime.it.}
}
\end{abstract}


\section{Introduction}\label{sec:1}
Let $R=K[x_{ij}, 1\le i \le m, 1\le j \le n]$ be the polynomial ring in $m\times n$ variables with coefficient in a field $K$ with the standard grading, \emph{i.e.}, $\deg(x_{ij})=1$. Let $G$ be a simple graph with $n$ vertices. In \cite{JR}, Rauh introduced the notion of generalized binomial edge ideal $J_{G, m}$ of $R$ associated to $G$ for studying conditional independence ideals. 
Such an ideal generalizes the binomial edge ideal introduced in \cite{HHHKR, MO}. Indeed, for $m=2$, $J_{G, m}=J_G$ is the classical binomial edge ideal of the graph $G$. Afterwards, in \cite{MK2,EHHQ}, the authors introduced the binomial edge ideal of a pair of graphs which is a generalization of both the generalized binomial edge ideals and the ideals generated by adjacent minors. More in details, if $G_1$ and $G_2$ are simple graphs with $m$ and $n$ vertices, respectively, and $f=\{i, j\}\in E(G_1)$, $g=\{k, \ell\}\in E(G_2)$, one can assign the $2$--minor $p_{f,g} =  x_{ik}x_{j\ell} - x_{i\ell}x_{jk}$ to the pair $(f, g)$. The binomial edge ideal of the pair $(G_1, G_2)$ is the ideal $J_{G_1, G_2} =(p_{f,g} : (f, g)\in E(G_1)\times E(G_2))$. One can observe that if $(G_1, G_2) =(K_m, G)$, where $K_m$ is the complete graph on $m$ vertices, $J_{K_m, G}$ is the generalized binomial edge ideal associate to $G$. In \cite{EHHQ} and \cite{CI}, unmixedness and Cohen–Macaulayness of binomial edge ideal of a pair of graphs are characterized in some special cases.\\
The algebraic properties of the class of generalized binomial edge ideals are widely open, although some results are known (see for instance \cite{CI, AK, CMM} and the reference therein). Furthermore, some problems solved for binomial edge ideals (see, for instance, \cite{KS,CR,HR}) could be dealt also for the class of generalized bonimial edge ideals.

It is a widely open problem, studied by many authors, the classification of Cohen--Macaulay binomial edge ideals (see for instance \cite{EHH}). Nevertheless, some recent papers \cite{BN,BMS1,BMS,LMRR,R2} relate the Cohen--Macaulyness to combinatorial properties of the graphs. So, the following question naturally arises:
\begin{open}
Is the Cohen--Macaulyness of (generalized) binomial edge ideals field independent?
\end{open}
In the case $m>2$, we give a positive answer to the question. In particular, we prove that the only Cohen--Macaulay graphs are the complete ones.

Our main tool is the \emph{dual graph} in the sense of  \cite{BV}. Such a notion was firstly introduced by Hartshorne in \cite{RHA}. He proved that the connectedness of the dual graph translates in combinatorial terms the notion of connectedness in codimension one \cite[Proposition 1.1]{RHA}. As a consequence one has that if $I$ is an ideal of a polynomial ring $A=K[x_1, \ldots, x_n]$, then if $A/I$ satisfies the Serre’s condition ($S_2$) or in particular, if $A/I$ is Cohen--Macaulay, then the dual graph $D(I)$ is connected. The converse does not hold in general (\cite[Remark 5.1]{BMS1}).\\
It is well--known that all Cohen--Macaulay ideals are unmixed. We observe that an unmixed generalized binomial ideal has an underlying graph that is $(m-1)$--connected.

A nice family of biconnected graphs are the circulant graphs.
We obtain a characterization of the unmixedness of the generalized binomial edge ideal associated to a non--complete power cycle. The power cycles have been studied in \cite{BH} and \cite{MTW}, and in the latter paper the authors determined a classification of Cohen--Macaulay power cycles edge ideal.

The paper is organized as follows. Section \ref{sec:2} contains some preliminary notions that will be used in the paper. In Section \ref{sec:3}, we analyze the unmixedness of the class of generalized binomial
edge ideal in some special cases.
The notion of cutsets plays a fundamental role.
They are fundamental in deepening the behavior of generalized binomial edge ideals since they are related to their minimal primary decomposition. Our main result is the characterization of the unmixedness of $J_{G, m}$ for $G$ non--complete power cycle (Theorem \ref{thm:main}).
We characterize the bipartite unmixed ones, too. 
We also give a computational classification of all unmixed generalized binomial edge ideals for any $m$ with $n\leq 10$. The implementation is freely downloadable from the website \cite{ACR} (see Table~\ref{tab:un}). 
In Section \ref{sec:4}, we state a characterization of the Cohen–-Macaulayness of the generalized binomial edge ideal $J_{G, m}$ for $m\ge 3$ (Corollary \ref{cor:main}). The key result is Lemma \ref{lem:dualempty}.

\section{Preliminaries}\label{sec:2}
 Let $G$ be a finite simple graph with vertex set $V(G)$, $|V(G)|=n$, and edge set $E(G)$. If $T$ is a subset of vertices of $V(G)$, we denote with $G\setminus T$ the induced subgraph of $G$ on $[n]\setminus T$ \cite{RV}. A graph $G$ is \emph{$k$--vertex--connected} (or simply \emph{$k$--connected}) if $k<n$ and for every subset $T$ of vertices such that $|T|<k$, the induced graph $G\setminus T$ is connected. In particular, a graph is \emph{biconnected} if it is $2$--connected.
We will denote
 by $C_n$ the cycle on $n$ vertices and by $K_n$  the complete graph on $n$ vertices.\\
 A set $T\subset  V (G)$ is called a \emph{cutset} of $G$ if $T= \emptyset$ or $c(T \setminus \{v\}) < c(T)$ for each $v\in T$, where $c(T)$ denotes the number of connected components induced by removing $T$ from $G$. We denote by $\mathcal{C}(G)$ the set of all cutsets of $G$. When $T\in \mathcal{C}(G)$ consists of one vertex $v$, $v$ is called a \emph{cutpoint}.\\
Let us recall the notion of a \emph{circulant graph}. Let $S\subseteq \{1,2, \ldots, \lfloor \frac n2 \rfloor\}$. The \emph{circulant graph} $G : = C_n(S)$ is a
simple graph with $V(G) = \ZZ_n = \{0, \ldots, n-1\}$ and $E(G) := \{\{i, j\} :  |j - i|_n \in S\}$ where $|k|_n = \min\{|k|, n - |k|\}$. $C_n(S)$ is the circulant graph of order $n$ with generating set $S$ and $|k|_n$ is the circular distance modulo $n$.
It is well--known that if the generating set $S = \{\pm 1,\pm 2, \ldots, \pm d\}$, where $1\le d \le \lfloor \frac n2 \rfloor$ is a given integer, then the circulant graph $C_n(S)$ is then equivalent to the \emph{$d$-th power} of $C_n$, where two vertices are adjacent if and only if their distance is at most $d$  \cite{RHO}.
In such a case, we will denote $C_n(S)$ by $C_n(1,2, \ldots, d)$. $C_n(1,2, \ldots, d)$ is called \emph{$d$-th power cycle}.\\
We close the section discussing the \emph{generalized binomial edge ideals}. Such a class has been introduced in \cite{JR}. Let $m, n \ge 2$ be integers and let $G$ be an arbitrary simple graph on the vertex set $[n]$.
Let $X= (x_{ij})$ be an $(m\times n)$--matrix of indeterminates, and denote by $R = K[X]$ the polynomial ring in the variables
$x_{ij}$, $i = 1, \ldots, m$ and $j=1, \ldots, n$. The generalized binomial edge ideal $J_G$ of $G$ is generated by
all the $2$--minors of $X$ of the form $[k, \ell\vert i, j] = x_{ki}x_{\ell j} - x_{kj}x_{\ell i}$, where $1\le k\le m$ and $\{i, j\}$ is an edge of $G$ with $i<j$.
When $m = 2$, $J_G$ coincides with the classical binomial edge ideal $J_G$ introduced in \cite{HHHKR, MO}. \\
Through the paper, in order to highlight the integer $m$, we will denote it by $J_{G, m}$.

\section{Unmixedness}\label{sec:3}

In this section we study the unmixedness of classes of generalized binomial edge ideal \emph{via} the cutsets.

Throughtout this paper we assume that $G$ is a connected graph on $[n]$ and let $\mathcal{C}(G)$ be the set of all cutsets of $G$. Let $T \in \mathcal{C}(G)$ and let $G_1, \dots, G_{c(T)}$ denote the connected components of the induced graph $G\setminus T$. 

Let us consider the ideal
\[
P_T(G) = \left( \bigcup_{i \in T} \{x_{1,i},\ldots, x_{m,i}\}, J_{\tilde{G}_1, m}, \dots, J_{\tilde{G}_{c(T)},m} \right)
\]
of the polynomial ring $R$, where $\tilde{G}_i$ is the complete graph on $V(G_i)$, for $i=1, \dots, c(T)$. From \cite[Section 3]{AK} (see also \cite{JR, EHHQ}), one has that
\begin{equation}\label{Eq:primarydec}
J_{G, m} = \bigcap_{T \in \mathcal{C}(G)} P_T(G). 
\end{equation}

The next result comes essentially from \cite[Proposition 4.1]{EHHQ}. We reformulate it for our purpose. 

\begin{lemma}\label{lem:unmixed}
Let $G$ be a graph with $n$ vertices. Then the generalized binomial edge ideal $J_{G, m}$ is unmixed  if and only if for all cutset $T$ of $G$, one has $$c(T)=\frac{|T|}{m-1}+1.$$
\end{lemma}

\begin{proof}
Since $J_{G,m}$ is unmixed then $\height J_{G,m}=\height P_\emptyset=(m-1)(n-1)$. Moreover 
\[
 \height P_T=(m-1)\sum_{i=1}^{c(T)} (n_i-1) +m|T|
\]
where $n_i=|G_i|$ and $G\setminus T=\sqcup_{i=1}^{c(T)} G_i$.
One can easily observe that 
\[
 \sum_{i=1}^{c(T)} (n_i-1)=n-|T|-c(T).
\]
Using the above equations the assertion follows.\\
The converse holds by the same arguments.
\end{proof}

\begin{remark}\label{rem:unmixed}
Under the same hypotheses of Lemma \ref{lem:unmixed}, for all cutset $T\neq \emptyset$ of $G$, one has that $m-1$ divides $|T|$.
\end{remark}

As a consequence we obtain the next result.

\begin{corollary} \label{cor:biconnected} Let $G$ be a graph with $n$ vertices. If $J_{G, m}$ is an unmixed generalized binomial edge ideal with $m\geq 3$, then
 \begin{enumerate}
  \item[\em(1)] $G$ is $(m-1)$--connected;
  \item[\em(2)] $\deg(v)\geq m-1$ for all $v\in V(G)$.
  \item[\em(3)] $G$ is complete if $m\geq n$.
 \end{enumerate}
\end{corollary}
\begin{proof}
 (1) It follows by Lemma \ref{lem:unmixed}. (2) See \cite[Theorem 5.1]{HARARY} . (3) By (2) the assertion follows.
\end{proof}

The previous results allow us to state a classification of three special classes of unmixed generalized binomial edge ideal $J_{G, m}$ for $m\geq 3$. Moreover, at the end of this section we present a table that shows the cardinality of the set of any unmixed $J_{G,m}$ with $n\leq 10$.

Let $A$ be a set of vertices of a graph $G$. The \emph{neighbor set}
of $A$, denoted by $N_G(A)$ or simply by $N(A)$ if $G$ is understood, is the set
of vertices of $G$ that are adjacent with at least one vertex of $A$. Moreover, if $A = \{v\}$, we denote by $N[v] = N(v)\cup \{v\}$.

\begin{proposition}\label{prop:quasi}
Let $n>3$ and let $G$  be a graph on $n$ vertices such that $\deg v\geq n-2$, for all $v\in V(G)$. Then $J_{G,n-1}$ is unmixed.
 
\end{proposition}
\begin{proof}
We claim that if $T$ is a non--empty cutset, then $T=N(v)$ with $\deg(v)=n-2$. \\
If  $\deg v=n-2$, then $G\setminus N(v)$ has two components: the vertex $v$ and the vertex $w$ which is the only one vertex not adjacent to $v$.

Since the graph has $n$ vertices, then the cutsets have the same cardinality $n-2$. In fact, if $T$ is a cutset then $T=N(v)=N(w)$, where $\deg(v)=\deg(w)=n-2$ and $\{v,w\}\notin E(G)$. The claim follows.  Finally, from Lemma \ref{lem:unmixed}, $J_{G,n-1}$ is unmixed.
\end{proof}

\begin{example}\label{ex:nfour}
Let $n=4$. There are only three graphs satisfying the hypotheses of Proposition~\ref{prop:quasi}. They are the following:

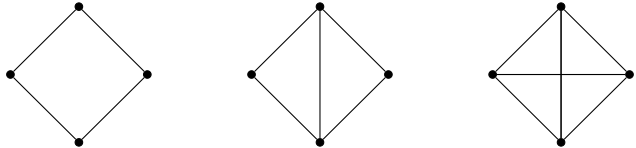
\begin{figure}[H]
\begin{center}
\begin{tikzpicture} [rotate=90,scale=0.6, vertices/.style={draw, fill=black, circle, inner sep=1.0pt}]
\def \n {4}
\def \radius {1.5cm}
\pgfmathtruncatemacro\nv{\n-1}
\def \S {1}
\foreach \v in {0,...,\nv}{
   \node[vertices] (\v) at ({-360/\n * \v}:\radius) {};
}
\foreach \v in {0,...,\nv}{
   \foreach \s in \S{
      \pgfmathtruncatemacro\m{mod(\v+\s,\n)}
      \path (\v) edge node[pos=0.5,below] {} (\m);
   }
}
\end{tikzpicture}
\hspace{1cm}
\begin{tikzpicture} [rotate=90,scale=0.6, vertices/.style={draw, fill=black, circle, inner sep=1.0pt}]
\def \n {4}
\def \radius {1.5cm}
\pgfmathtruncatemacro\nv{\n-1}
\def \S {1}
\foreach \v in {0,...,\nv}{
   \node[vertices] (\v) at ({-360/\n * \v}:\radius) {};
}
\foreach \v in {0,...,\nv}{
   \foreach \s in \S{
      \pgfmathtruncatemacro\m{mod(\v+\s,\n)}
      \path (\v) edge node[pos=0.5,below] {} (\m);
   }
}
\path (0) edge node[pos=0.5,below] {} (2);
\end{tikzpicture}
\hspace{1cm}
\begin{tikzpicture} [rotate=90,scale=0.6, vertices/.style={draw, fill=black, circle, inner sep=1.0pt}]
\def \n {4}
\def \radius {1.5cm}
\pgfmathtruncatemacro\nv{\n-1}
\def \S {1,2}
\foreach \v in {0,...,\nv}{
   \node[vertices] (\v) at ({-360/\n * \v}:\radius) {};
}
\foreach \v in {0,...,\nv}{
   \foreach \s in \S{
      \pgfmathtruncatemacro\m{mod(\v+\s,\n)}
      \path (\v) edge node[pos=0.5,below] {} (\m);
   }
}
\end{tikzpicture}
\end{center}
\end{figure}
\end{example}

\begin{proposition}
Let $G$ be a bipartite graph with $n$ vertices and let $J_{G,m}$ be an unmixed ideal with $m\geq 3$.
 \begin{enumerate}
  \item[\em (1)] If $m=3$, then $G\in\{C_4,K_2\}$.
  \item[\em (2)] if $m>3$, then $G=K_2$.
 \end{enumerate}
\end{proposition}
\begin{proof}
Let $G$ be a bipartite graph with vertex set $V_1\sqcup V_2$. Firstly, we observe that since $m\geq 3$, then $G$ is biconnected (Corollary \ref{cor:biconnected}). Furthermore,  if $G$ is not $K_2$, then $V_1,V_2\in \CC(G)$.
In fact, $G\setminus V_1$, respectively $G\setminus V_2$, is the set of isolated vertices $V_2$, respectively $V_1$. If we consider $V_1\setminus \{v\}$ then $c\left(V_1\setminus \{v\}\right)>c(V_1)$, indeed each $v\in V_i$ has degree at least $2$.
Moreover,  
\begin{eqnarray*}
 |V_1|&=&(m-1)(|V_2|-1),\\
 |V_2|&=&(m-1)(|V_1|-1).
\end{eqnarray*}
Note that the only integer solutions for these equations are obtained for $m=3$ and $|V_1|=|V_2|=2$.
\end{proof}

\begin{remark}
We refer to the paper \cite{BMS} for bipartite graphs when $m=2$.
\end{remark}

Now, we focus our attention on a family of circulant graphs. Recall that such graphs are biconnected.

\begin{lemma}\label{lem:powercut}
 Let $G=C_n(S)$ be a non--complete power cycle. Then $S\in\mathcal{C}(G)$.
\end{lemma}
\begin{proof} Let $G=C_n(1,2,\ldots,r)$, that is, $S=\{1,2,\ldots,r\}\cup \{n-r,n-r+1,\ldots,n-1\}$.
Since $G$ is non--complete, then $|S|$ is even and $r<\left\lfloor \frac{n}{2}\right\rfloor$. The vertex set of the subgraph induced by removing $S$ is $V(G\setminus S)=\{0\}\cup\{r+1,r+2,\ldots,n-r-1\}$. We observe that the set of vertices $\{r+1,r+2,\ldots,n-r-1\}$ is non--empty (from the non--completeness of $G$), and $0$ is an isolated vertex of $G\setminus S$.
Hence, the induced graph $G\setminus S$ has at least two connected components.
Moreover, $S$ is a cutset for $G$. Indeed, if $v\in S$ and $S'=S\setminus\{v\}$, then either $\{0,v\}, \{v,r+1\}\in E\left(G\setminus S'\right)$ or $\{0,v\}, \{v,n-r-1\}\in E\left(G\setminus S'\right)$, that is, $0$ is not isolated in the  subgraph induced by removing $S'$. This implies that $c(S')<c(S)$ and $S\in\mathcal{C}(G)$.
\end{proof}

\begin{remark}\label{rem:powercut}
Let $G=C_n(1,2,\ldots,r)$ be a non--complete power cycle. In the proof of Lemma~\ref{lem:powercut}, we state that $G\setminus S$ has at least two connected components. More precisely, $G\setminus S$ has exactly two connected components: $\{0\}$ and a connected component $K$ with vertices $\{r+1,r+2,\ldots,n-r-1\}$. 
The connectedness of $K$ is assured by the existence of a path from $r+1$ to $n-r-1$, since $1\in S$. In general, $K$ is not a complete subgraph.
\end{remark}

\begin{corollary}\label{cor:power1}
Let $G=C_n(S)$ be a non--complete power cycle and let $m$ be an integer greater than $1$. 
If $J_{G,m}$ is unmixed, then $|S|=m-1$. 
\end{corollary}
\begin{proof}
From Lemma~\ref{lem:powercut} and Remark~\ref{rem:powercut} we have that $S\in\mathcal{C}(G)$ and $c(S)=2$. 
Moreover, from Lemma~\ref{lem:unmixed} we have that $2=\frac{|S|}{m-1}+1$. The assertion follows.
\end{proof}

\begin{example}
In general, if $G=C_n(S)$ is not a power cycle, then the result of Corollary~\ref{cor:power1} does not hold. Indeed, we can consider the circulant graph $G=C_8({1,3,4})$. In such a case, $S=\{1,3,4,5,7\}$. The generalized binomial edge ideal $J_{G,5}$ is unmixed and $|S|\neq 4$.

Moreover, we observe that $S$ is not a cutset for $G$. Indeed, $S$ properly contains the cutset $T=\{1,3,5,7\}$ for $G$.
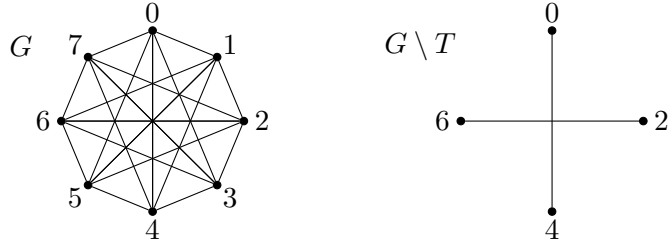
\begin{figure}[H]
\begin{center}
\begin{tikzpicture} [rotate=90,scale=0.8, vertices/.style={draw, fill=black, circle, inner sep=1.0pt}]
\def \n {8}
\def \radius {1.5cm}
\pgfmathtruncatemacro\nv{\n-1}
\def \S {1,3,4}
\node (G) at ({60}:\radius+1cm) {$G$};
\foreach \v in {0,...,\nv}{
   \node[vertices] (\v) at ({-360/\n * \v}:\radius) {};
   \node at ({-360/\n * \v}:\radius+.3cm) {${\v}$};
}
\foreach \v in {0,...,\nv}{
   \foreach \s in \S{
      \pgfmathtruncatemacro\m{mod(\v+\s,\n)}
      \path (\v) edge node[pos=0.5,below] {} (\m);
   }
}
\end{tikzpicture}
\hspace{1cm}
\begin{tikzpicture} [rotate=90,scale=0.8,vertices/.style={draw, fill=black, circle, inner sep=1.0pt}]
\def \n {8}
\def \radius {1.5cm}
\pgfmathtruncatemacro\nv{\n-1}
\def \S {1}
\node (GS) at ({60}:\radius+1cm) {$G\setminus T$};
\foreach \v in {0,2,4,6}{
   \node[vertices] (\v) at ({-360/\n * \v}:\radius) {};
   \node at ({-360/\n * \v}:\radius+.3cm) {${\v}$};
}
\node (P) at ({180}:\radius+.3cm) {\phantom{4}};
\path (2) edge node[pos=0.5,below] {} (6);
\path (0) edge node[pos=0.5,below] {} (4);
\end{tikzpicture}
\caption{The circulant graph $G=C_8({1,3,4})$ and $G\setminus S$.}
\end{center}
\end{figure}
\end{example}

\begin{remark} Let $G=C_n(S)$ be a circulant graph such that $|S|$ is odd, then $n=2k$ and $k\in S$. Indeed, if $|S|$ is odd, then there exists $s\in S$ such that $s\equiv -s$. The assertion follows.
\end{remark}
Now, we will show some results in order to characterize the power cycles with unmixed generalized binomial edge ideal. 

\begin{lemma}\label{lem:evenpower}
Let $G=C_n(1,\ldots,r)$ be a non--complete power cycle and let $m$ be an even integer greater than $1$. 
Then $J_{G,m}$ is not unmixed.
\end{lemma}
\begin{proof}
From Corollary~\ref{cor:power1}, if we assume that $J_{G,m}$ is unmixed, then $|S|$ is odd. We can write $S=\{\pm 1, \pm 2,\ldots, \pm r \}$, hence, $r \equiv_{n} -r$. This implies that the power cycle is complete, against the assumptions. So, $J_{G,m}$ is not unmixed. 
\end{proof}

The next result characterizes all the unmixed ideals $J_{G,m}$, where $G$ is a non--complete power cycle $C_n(1,\ldots,r)$. We also show some examples to clarify the techniques used in the proof of the theorem.  

\begin{theorem}\label{thm:main}
Let $G= C_n(1,\ldots,r)$ be a non--complete power cycle.
Then $J_{G,m}$ is unmixed if and only if $m$ is odd, $n\in \{m+1,\ldots, \frac{3m+1}{2}\}$ and $r=\frac{m-1}{2}$.  
\end{theorem}
\begin{proof} Let $J_{G,m}$ be unmixed. Then, $m$ is odd (see Lemma~\ref{lem:evenpower}), $|S|=$ $|\{\pm 1, \pm 2,\ldots, \pm r \}|$ $=m-1$ (see Corollary~\ref{cor:power1}) and consequently $r=\frac{m-1}{2}$. Moreover, we observe that $S\cup \{0\}\subseteq \ZZ_n$, and so $m=|S|+1\leq n$. If we consider non--complete power cycles, then $n\geq m+1$. On the other hand, if $n>\frac{3m+1}{2}$ then it is possible to find at least a cutset $T$ such that $|T|\not \equiv_{m-1} 0$. In particular, $T$ has cardinality $3r=3\frac{m-1}{2}\not \equiv_{m-1} 0$.\\
\textsc{Claim 1.} If $n>\frac{3m+1}{2}$, then a cutset for $G$ is 
\begin{align}\label{eq:cutsetNoU}
T= & \{1,\ldots,r\} \cup \{r+2,\ldots,2r+1\} \cup \{2r+3,\ldots,3r+2\} \\
= & \bigcup_{k=0}^{2}{\left\{kr+(k+1),\ldots,(k+1)r+k\right\}}\nonumber.
\end{align}
\textsc{Proof of Claim 1.} We observe that in $G=C_n(1,\ldots,r)$, it is
\[
N(r+1)=\{1,\ldots,r\} \cup \{r+2,\ldots,2r+1\}.
\]
Indeed, the adjacencies of $r+1$ consist exactly of the \emph{previous} $r$ and the \emph{successive} $r$ vertices of $G$. In such a case, we identify the \emph{successive} of a vertex $v$ with $v+1$, and its \emph{previous} with $v-1$ in $\ZZ_n$. Analogously,
\[
N(2r+2)=\{r+2,\ldots,2r+1\} \cup \{2r+3,\ldots,3r+2\}.
\]
We note that the assumption of the claim implies $3r+2=3\frac{m-1}{2}+2=\frac{3m+1}{2}<n$. Hence, $G\setminus T$ has $3$ connected components. In fact, $r+1$ and $r+2$ are isolated vertices and there is a connected component containing $0$. Furthermore, $T$ is a cutset. We show that 
if $T'=T\setminus \{v\}$, then $G\setminus T'$ has $2$ connected components. Indeed, if $v\in \{1,\ldots,r\}$
there exists a path from $0$ to $r+1$. Hence, $r+1$ is not isolated.
Analogously, we obtain the same situation in the other cases. 
We note that if $n=\frac{3m+1}{2}$, then $T$ is not a cutset. In fact, $3r+2=\frac{3m+1}{2}\equiv_{n} 0$ and in such a case $c\left(T\right)=c\left(T\setminus\{0\}\right)$ (see Example~\ref{ex:pcnotunm} for more details).\\
Hence, for $n\not \in \{m+1,\ldots, \frac{3m+1}{2}\}$, there do not exist non--complete unmixed power cycles, against the assumption. So, we have $n \in \{m+1,\ldots, \frac{3m+1}{2}\}$.\\

Conversely, suppose $m$ odd, $n\in \{m+1,\ldots, \frac{3m+1}{2}\}$ and $r=\frac{m-1}{2}$. We have to prove that the ideal $J_{G,m}$, $G=C_n(1,\ldots,r)$, is unmixed. This means that if we fix $m$ odd, then there exist exactly $r+1$ non--complete unmixed power cycles.  
Given a power cycle as before, we will show that all its cutsets  $T$ share the same cardinality, that is, $2r=2\frac{m-1}{2}=m-1$. If it holds true, then $\frac{|T|}{m-1}+1=2$ for all cutsets $T$ and the ideal $J_{G,m}$ is unmixed (Lemma~\ref{lem:unmixed}).
 
By the hypotheses about $n$, we can write $n=m+i$, for $i=1,\ldots,r+1$. We will prove
that the number of cutsets of $C_{m+i}(1,\ldots,r)$ is $\frac{(m+i)i}{2}=\frac{ni}{2}$, for $i=1,\ldots,r+1$. We observe that, for the oddity of $m$, one among $i$ and $n$ must be even.
We describe these cutsets.\\
\textsc{Claim 2.} We prove that if $G=C_{m+i}(1,\ldots,r)$, then $\mathcal{C}(G)=\{T_{j,k}\}\cup\{\emptyset\}$, where 
\[
T_{j,k}=\{j,\ldots,r+j-1\} \cup \{r+j+k,\ldots,r+j+k+r-1\},
\]
for $j=1,\ldots,r+i$ and $k=1,\ldots,q$, with
\[
q=\begin{cases}
i & \text{if } j\leq r+1, \\
i-j+r+1 & \text{otherwise}.
\end{cases}
\]
All the elements of the sets $T_{j,k}$ are up congruence module $n$. Moreover, we call two vertices $v,w$ \emph{consecutive} if  either $v=w+1$ or $v=w-1$ in $\ZZ_n$.\\
\textsc{Proof of Claim 2.}
Firstly, we observe that $S\in\mathcal{C}(G)$ (Lemma~\ref{lem:powercut}).
Furthermore, we note that $2r+i=m-1+i=n-1$ is not a \emph{consecutive} vertex of $1=j$. So, we can write $S=T_{1,i}=\{1,\ldots,r\} \cup \{r+1+i,\ldots,2r+i\} = \{1,\ldots,r\} \cup \{n-r,\ldots,n-1\}$.\\
Now, we consider $j=1$ and we describe 
the set of vertices $T_{1,k}$ for $k=1,\ldots,i$:
\[
\begin{array}{ccl}
\{1,\ldots ,r\} & \cup & \{r+2,r+3,\ldots,2r+1\},\\
\{1,\ldots ,r\} & \cup & \{r+3,r+4,\ldots,2r+2\},\\
\vdots & \vdots & \phantom{\{r+3,r+4,} \vdots\\
\{1,\ldots ,r\} & \cup  & \{r+i+1,r+i+2,\ldots,2r+i\}.
\end{array}
\]
If we consider $j=2$, we can write the sets $T_{2,k}$ for $k=1,\ldots,i$:
\[
\begin{array}{ccl}
\{2,\ldots ,r+1\} & \cup & \{r+3,r+4,\ldots,2r+2\},\\
\vdots & \vdots & \phantom{\{r+3,r+4,} \vdots\\
\{2,\ldots ,r+1\} & \cup & \{r+i+1,r+i+2,\ldots,2r+i\},\\
\{2,\ldots ,r+1\} & \cup & \{r+i+2,r+i+3,\ldots,2r+i+1\}.
\end{array}
\]
We note that $2r+i+1=m-1+i+1=n=0$ is not a \emph{consecutive} vertex of $2=j$.\\
We can iterate this construction until we reach the value $j=r+1$. So, we construct the sets $T_{r+1,k}$:
\[
\begin{array}{ccl}
\{r+1,\ldots ,2r\} & \cup & \{2r+2,2r+3,\ldots,3r+1\},\\
\vdots & \vdots & \phantom{\{2r+2,2r+3,} \vdots\\
\{r+1,\ldots ,2r\} & \cup & \{2r+i+1,2r+i+2,\ldots,3r+i\}.
\end{array}
\]
Also in this case, we have that $3r+i=r+m-1+i=r-1$ is not a \emph{consecutive} vertex of $r+1=j$.\\
Furthermore, we observe that all the sets of vertices $T_{j,k}$, with $j=1,\ldots,r+1$ and $k=1,\ldots,i$, are distinct (see Example~\ref{ex:pcunm}).

Now, we show that $T_{j,k}\in\mathcal{C}(G)$. Firstly, we observe that if $k=1$ then $T_{j,k}=N(r+j)$, and if $k=i$ then $T_{j,k}=N(j-1)$. In both these cases, $T_{j,k}$ is a cutset (Lemma~\ref{lem:powercut}) for each value of $j$. More in detail, $G\setminus T_{j,k}$ has two connected components that are complete subgraphs of $G$: an isolated vertex ($r+j$ or $j-1$) and the complete graph $K_i$ ($i=m+i-2r-1$).
Indeed, from $n\leq\frac{3m+1}{2}=3r+2$ we have $i=m+i-2r-1\leq 3r+2-2r-1=r+1$. This implies that each vertex of $K_i$ is adjacent to all the remaining $r$ vertices.  

Fixing $j=1$ and $k=2$, we observe that $G\setminus T_{1,2}$ has two connected components that, using similar arguments as before, are two complete subgraphs of $G$: $K_2$ (with vertices $\{r+1,r+2\}$) and $K_{i-1}$ (with the vertices of $G$ not in $T_{1,2}\cup \{r+1,r+2\}$). Now, we consider $T'=T_{1,2}\setminus \{v\}$. If $v\in \{1,\ldots,r\}$ then
there exists a path from $0$ to $r+1$, and so $G\setminus T'$ is connected. If $v\in \{r+3,\ldots,2r+2\}$ then there exists a path from $r+2$ to $2r+3$, and $G\setminus T'$ is connected as before. Hence, $T_{1,2}\in \mathcal{C}(G)$. This holds true for each $1<k<i$, applying an analogous reasoning.

In general, for symmetry, this result holds for all $j=1,\ldots, r+i$ and for all $k=1,\ldots, i$. We have distinguished  
the cases $j\leq r+1$ 
and $j\geq r+2$ in order to write all the cutsets without repetition. Indeed, when $j=r+2$ and $k=i$ we obtain the cutset
\begin{align*}
T_{r+2,i}= & \{r+2, r+3, \ldots, 2r+1\} \cup \{2r+2+i,\ldots,3r+1+i\}\\
= & \{r+2, r+3, \ldots, 2r+1\} \cup \{1,\ldots,r\}\\
= & T_{1,1}
\end{align*}
($2r+2+i=m+1+i=1$ and $3r+1+i=r+m+i=r$). In such a case, it is sufficient for $k$ to assume values between $1$ and $i-1$.\\
When $j=r+i$, we have only a new cutset, $T_{r+i,1}$, since all the others have already been considered, that is, the maximum value for $k$ can be $1$. Finally, when $j=r+i+1$ no new cutset $T_{j,k}$ can be found.

For what has been said so far, we are sure that the cutsets $T_{j,k}$ have the minimal cardinality among all possible cutsets. Furthermore, there no exist cutsets of cardinality greater than $2r$. Indeed, we have already observed that $G\setminus T_{j,k}$ has always two connected components that are complete subgraphs of $G$. Since a complete subgraph has the empty cutsets only, then the number of connected components of $G\setminus T'$, with $T'=T_{j,k}\cup\{v\}$, is the same as that of $G\setminus T_{j,k}$. Hence, $\mathcal{C}(G)=\left\{ T_{j,k} \right\}$ for $j=1,\ldots,r+i$ and $k=1,\ldots,q$, as before defined. 

Finally, we count the number of cutsets of $G=C_{m+i}(1,\ldots,r)$. To do this, we need to distinguish two cases. When $j\leq r+1$, the index $k$ has a constant range, from $1$ to $i$. So, we have $i(r+1)$ cutsets. When $j \geq r+2$, the index $k$ can assume values from $1$ to $i-j+r+1$. For the last value of $j=r+i$, we have $k=1$. This means that the number of the cutsets is the sum of the first $i-1$ positive integer ($r+i-r-2+1=i-1$), that is, $\frac{i(i-1)}{2}$.
Hence, we have that
\[
|\mathcal{C}(G)|=i(r+1)+\frac{i(i-1)}{2}=i\left(\frac{m+1+i-1}{2}\right)=\frac{ni}{2}.
\]
\end{proof}

Now, we give some examples in order to clarify the techniques used in the proof of Theorem~\ref{thm:main}. They are tested also using \emph{Macaulay2} \cite{Ma}.

\begin{example}\label{ex:pcnotunm}
Let $G=C_{12}(1,2,3)$ and $m=7$. In such a case, we have $r=3$ and $i=5$. From Theorem~\ref{thm:main}, $J_{G,m}$ is not unmixed: $12>\frac{3m+1}{2}=11$. Indeed, by Formula~(\ref{eq:cutsetNoU}), we can find a cutset $T$ such that $|T|=9\not \equiv_6 0$:
\[
T=\{1,2,3\} \cup \{5,6,7\} \cup \{9,10,11\}.
\]

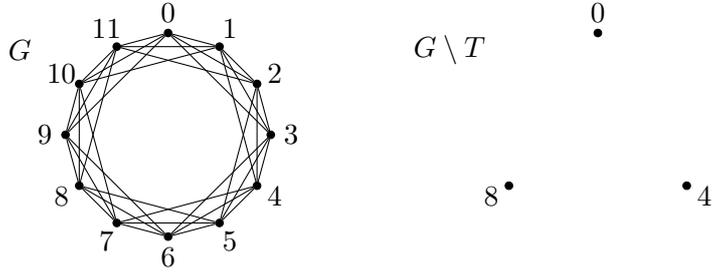
\begin{center}
\begin{tikzpicture} [rotate=90,scale=0.9, vertices/.style={draw, fill=black, circle, inner sep=1.0pt}]
\def \n {12}
\def \radius {1.5cm}
\pgfmathtruncatemacro\nv{\n-1}
\def \S {1,2,3}
\node (G) at ({60}:\radius+1cm) {$G$};
\foreach \v in {0,...,\nv}{
   \node[vertices] (\v) at ({-360/\n * \v}:\radius) {};
   \node at ({-360/\n * \v}:\radius+.3cm) {${\v}$};
}
\foreach \v in {0,...,\nv}{
   \foreach \s in \S{
      \pgfmathtruncatemacro\m{mod(\v+\s,\n)}
      \path (\v) edge node[pos=0.5,below] {} (\m);
   }
}
\end{tikzpicture}
\hspace{1cm}
\begin{tikzpicture} [rotate=90,scale=0.9,vertices/.style={draw, fill=black, circle, inner sep=1.0pt}]
\def \n {12}
\def \radius {1.5cm}
\pgfmathtruncatemacro\nv{\n-1}
\def \S {1,2,3}
\node (GT) at ({60}:\radius+1cm) {$G\setminus T$};
\foreach \v in {0,4,8}{
   \node[vertices] (\v) at ({-360/\n * \v}:\radius) {};
   \node at ({-360/\n * \v}:\radius+.3cm) {${\v}$};
}
\node (P) at ({180}:\radius+.3cm) {\phantom{6}};
\end{tikzpicture}
\end{center}
$G\setminus T$ has $3$ connected components, and, from the characterization in Lemma~\ref{lem:unmixed}, $J_{G,m}$ is not unmixed.
\end{example}

\begin{example}\label{ex:pcunm}
Let $G=C_8(1,2)$ and $m=5$. In such a case, we have $r=2$ and $i=3$. Theorem~\ref{thm:main} assures that $J_{G,m}$ is unmixed. Firstly, we observe that $S=\{1,2,6,7\}$. 

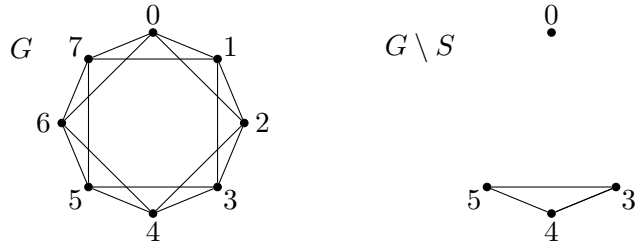
\begin{center}
\begin{tikzpicture} [rotate=90,scale=0.8, vertices/.style={draw, fill=black, circle, inner sep=1.0pt}]
\def \n {8}
\def \radius {1.5cm}
\pgfmathtruncatemacro\nv{\n-1}
\def \S {1,2}
\node (G) at ({60}:\radius+1cm) {$G$};
\foreach \v in {0,...,\nv}{
   \node[vertices] (\v) at ({-360/\n * \v}:\radius) {};
   \node at ({-360/\n * \v}:\radius+.3cm) {${\v}$};
}
\foreach \v in {0,...,\nv}{
   \foreach \s in \S{
      \pgfmathtruncatemacro\m{mod(\v+\s,\n)}
      \path (\v) edge node[pos=0.5,below] {} (\m);
   }
}
\end{tikzpicture}
\hspace{1cm}
\begin{tikzpicture} [rotate=90,scale=0.8,vertices/.style={draw, fill=black, circle, inner sep=1.0pt}]
\def \n {8}
\def \radius {1.5cm}
\pgfmathtruncatemacro\nv{\n-1}
\def \S {1,2}
\node (GS) at ({60}:\radius+1cm) {$G\setminus S$};
\foreach \v in {0,3,4,5}{
   \node[vertices] (\v) at ({-360/\n * \v}:\radius) {};
   \node at ({-360/\n * \v}:\radius+.3cm) {${\v}$};
}
\foreach \v in {0,...,1}{
   \foreach \s in \S{
      \pgfmathtruncatemacro\m{mod(\v+\s,3)+3}
      \pgfmathtruncatemacro\k{\v+3}
      \path (\k) edge node[pos=0.5,below] {} (\m);
   }
}
\end{tikzpicture}
\end{center}

All the cutsets of $G$ are the $T_{j,k}$ for $j=1,\ldots,6$ and $k=1,\ldots,3$ as follows:
\[
\begin{array}{ccc}
\begin{array}{c}
T_{1,1}=\{1,2\} \cup \{4,5\},\\
T_{1,2}=\{1,2\} \cup \{5,6\},\\
T_{1,3}=\{1,2\} \cup \{6,7\};
\end{array}
&
\begin{array}{c}
T_{2,1}=\{2,3\} \cup \{5,6\},\\
T_{2,2}=\{2,3\} \cup \{6,7\},\\
T_{2,3}=\{2,3\} \cup \{7,0\};
\end{array}
&
\begin{array}{c}
T_{3,1}=\{3,4\} \cup \{6,7\},\\
T_{3,2}=\{3,4\} \cup \{7,0\},\\
T_{3,3}=\{3,4\} \cup \{0,1\};
\end{array}\\
& & \\
\begin{array}{c}
T_{4,1}=\{4,5\} \cup \{7,0\},\\
T_{4,2}=\{4,5\} \cup \{0,1\},\\
T_{4,3}=\cancel{\{4,5\} \cup \{1,2\}};
\end{array}
&
\begin{array}{c}
T_{5,1}=\{5,6\} \cup \{0,1\},\\
T_{5,2}=\cancel{\{5,6\} \cup \{1,2\}},\\
T_{5,3}=\cancel{\{5,6\} \cup \{2,3\}};
\end{array}
&
\begin{array}{c}
T_{6,1}=\cancel{\{6,7\} \cup \{1,2\}},\\
T_{6,2}=\cancel{\{6,7\} \cup \{2,3\}},\\
T_{6,3}=\cancel{\{6,7\} \cup \{3,4\}}.
\end{array}
\end{array}
\]
We can note that $T_{6,k}$, for $k=1,\ldots,3$, does not add new cutsets of $G$. Moreover, for $j>3$, the repetition of the cutsets \emph{gradually} begins (see the proof of Theorem~\ref{thm:main}). Furthermore, the cardinality of each cutset is $4$. 
Hence, we have that $|\mathcal{C}(G)|=\frac{in}{2}=12$.
\end{example}

Modifying the algorithm implemented in
\cite{LMRR}  for the unmixedness of binomial edge ideals we computed all the unmixed $J_{G,m}$ with $n\leq 10$ vertices.  The implementation and results are downloadable from \cite{ACR}.\\
In the following table we show the cardinalities of the sets of unmixed $J_{G,m}$ on $n$ vertices.
\begin{table}[H]
\begin{center}
\begin{tabular}{p{6pt}|r|rrrrrrr|}
 \multicolumn{9}{c}{\phantom{0000000}(numbers of vertices $n$)} \\
\cline{3-9}
\multicolumn{2}{c|}{} & 4 &  5 & 6 & 7 & 8 & 9 & 10 \\
\cline{2-9}
\multirow{10}{*}{\rot{\phantom{0000}(values of $m$)}} 
& 3  & \bf 3 & 5 & 8 & 20 & 61 & 153 & 496 \\
& 4  & 1 & 3 & 8 & 12 & 20 & 52  & 141 \\
& 5  & 1 & 1 & 4 & 11 & 33 & 44  & 73 \\
& 6  & 1 & 1 & 1 & 4  & 16 & 51  & 111 \\
& 7  & 1 & 1 & 1 & 1  & 5  & 23  & 126 \\
& 8  & 1 & 1 & 1 & 1  & 1  & 5   & 33 \\
& 9  & 1 & 1 & 1 & 1  & 1  & 1   & 6 \\
& 10 & 1 & 1 & 1 & 1  & 1  & 1   & 1 \\
\cline{2-9}
\end{tabular}
\caption{\label{tab:un}Cardinalities of the sets of unmixed $J_{G,m}$ on $n$ vertices.}
\end{center}
\end{table}

The ideals $J_{G,m}$ counted by the highlighted entry $(3,4)$ of the Table~\ref{tab:un} are exactly the ones related to the graphs in Example~\ref{ex:nfour}.

\section{Cohen--Macaulayness}\label{sec:4}

Let $J$ be an ideal in a polynomial ring $S = K[x_1,\ldots,x_n]$ and let $P_1 ,\ldots,P_r$ be the minimal prime ideals of $J$. The dual graph $D(J)$ of $J$ is the graph with vertex set $V\left(D(J)\right)=\{1,\ldots,r\}$ and edge set
\[
E\left(D(J)\right)=\left\{\{i,j\}:\height(P_i + P_j )-1=\height(P_i)=\height(P_j)=\height(J)\right\}.
\] 

\begin{lemma}\label{lem:dualempty}
 Let $J_{G,m}$ be a generalized binomial edge ideal with $m\geq 2$ such that  $G$ is a biconnected graph on $n$ vertices. Then $P_\emptyset$ is an isolated vertex in $D(J)$.
\end{lemma}

\begin{proof}
If $G$ is complete then $D(J)$ is the graph with only one vertex, that is, $P_\emptyset$. Hence, we assume $G$ is not complete. For any non--empty cutset $T$ we define $Q=P_T+ P_\emptyset$. We will show that the edge $\{P_\emptyset,P_T\}$ is not an edge of $D(G)$.
We recall that $P_\emptyset=I_2(X)$, namely the ideal of $2$-minors of the matrix $m\times n$, and its height is $(m-1)(n-1)$. 

Let $|T|=t$. Then after relabelling the vertices, we have that 
\[
 Q=\sum_{i=1}^m (x_{i,1},\ldots, x_{i,t}) + \sum_{\substack{1\leq i_1<i_2\leq m\\t+1\leq j_1<j_2\leq n}}(x_{i_1,j_1}x_{i_2,j_2}-x_{i_1,j_2}x_{i_2,j_1}).
\]

Hence $Q$ is prime and its height is $tm+(m-1)(n-t-1)$. The first addend comes from the monomial generators of $Q$ and the second one comes from by the binomial generators in $Q$. Moreover, the binomial generators of $Q$ determine the ideal $I_2(X')$, that is, the ideal of $2$-minors of the $m\times(n-t)$ matrix.
Therefore
\[
 \height Q-1= tm+(m-1)(n-t-1)-1=(m-1)(n-1)+t-1 > (m-1)(n-1),
\]
when $t>1$, namely when $G$ is biconnected. The assertion follows.
\end{proof}
The next result has been obtained in \cite{BN} using a different type of argument.
\begin{corollary}
Let $J_{G}$ be a binomial edge ideal such that  $G$ is biconnected. Then $J_G$ is Cohen--Macaulay if and only if $G$ is the complete graph. 
\end{corollary}

\begin{corollary}\label{cor:main}
Let $J_{G,m}$ be a generalized binomial edge ideal with $m\geq 3$. Then $J_{G,m}$ is Cohen--Macaulay if and only if $G$ is the complete graph. 
\end{corollary}
\begin{proof}
By Corollary \ref{cor:biconnected}, $G$ is biconnected. Thus the assertion follows by Lemma \ref{lem:dualempty}. 
\end{proof}

\bibliographystyle{abbrv}
\bibliography{ACR}

\begin{thebibliography}{10}

\bibitem{ACR}
L.~Amata, M.~Crupi, and G.~Rinaldo.
\newblock Unmixed generalized binomial edge ideals.
\newblock Available at
  http://www.giancarlorinaldo.it/unmixed-generalized-binomial-edge-ideals.html,
  2022.

\bibitem{BN}
A.~Banerjee and L.~N\'u\~{n}ez Betancourt.
\newblock Graph connectivity and binomial edge ideals.
\newblock {\em Proceedings of the American Mathematical Society}, 145:487--499,
  2017.

\bibitem{BV}
B.~Benedetti and M.~Varbaro.
\newblock On the dual graphs of {C}ohen-{M}acaulay algebras.
\newblock {\em International Mathematics Research Notices}, 2015:8085--8115,
  2015.

\bibitem{BMS1}
D.~Bolognini, A.~Macchia, and F.~Strazzanti.
\newblock Binomial edge ideals of bipartite graphs.
\newblock {\em European Journal of Combinatorics}, 70:1--25, 2018.

\bibitem{BMS}
D.~Bolognini, A.~Macchia, and F.~Strazzanti.
\newblock Cohen-{M}acaulay binomial edge ideals and accessible graphs.
\newblock {\em Journal of Algebraic Combinatorics}, 2021.

\bibitem{BH}
J.~Brown and R.~Hoshino.
\newblock Well-covered circulant graphs.
\newblock {\em Discrete Mathematics}, 311:244--251, 2011.

\bibitem{CI}
F.~Chaudhry and R.~Irfan.
\newblock On the generalized binomial edge ideals of generalized block graphs.
\newblock {\em Mathematical Reports}, 72:381--394, 2020.

\bibitem{CMM}
O.~Clarke, F.~Mohammadi, and H.~J. Motwani.
\newblock Conditional probabilities via line arrangements and point
  configurations.
\newblock {\em Linear and multilinear algebra}, 2021.
\newblock in press.

\bibitem{CR}
M.~Crupi and G.~Rinaldo.
\newblock Binomial edge ideals with quadratic gr\"obner bases.
\newblock {\em Electronic Journal of Combinatorics}, 18:P211, 2011.

\bibitem{EHH}
V.~Ene, J.~Herzog, and T.~Hibi.
\newblock {C}ohen-{M}acaulay binomial edge ideals.
\newblock {\em Nagoya Mathematical Journal}, 45:57--68, 2011.

\bibitem{EHHQ}
V.~Ene, J.~Herzog, T.~Hibi, and A.~A. Qureshi.
\newblock The binomial edge ideal of a pair of graphs.
\newblock {\em Nagoya Mathematical Journal}, 213:105--125, 2014.

\bibitem{Ma}
D.~R. Grayson and M.~E. Stillman.
\newblock Macaulay2, a software system for research in algebraic geometry.
\newblock Available at \url{http://www.math.uiuc.edu/Macaulay2/}, 2020.

\bibitem{HARARY}
F.~Harary.
\newblock {\em Graph Theory}.
\newblock Addison {W}esley series in mathematics. Addison-{W}esley, 1969.

\bibitem{RHA}
R.~Hartshorne.
\newblock Complete intersections and connectedness.
\newblock {\em American Journal of Mathematics}, 84:497--508, 1962.

\bibitem{HHHKR}
J.~Herzog, T.~Hibi, F.~Hreinsd\'ottir, T.~Kahle, and J.~Rauh.
\newblock Binomial edge ideals and conditional independence statements.
\newblock {\em Advances in Applied Mathematics}, 45:317--333, 2010.

\bibitem{HR}
J.~Herzog and G.~Rinaldo.
\newblock On the extremal betti numbers of binomial edge ideals of block
  graphs.
\newblock {\em Electronic Journal of Combinatorics}, 25:P1.63, 2018.

\bibitem{RHO}
R.~Hoshino.
\newblock {\em Independence Polynomials of Circulant Graphs}.
\newblock PhD thesis, Department of Mathematics and Statistics, Dalhouise
  University, 2008.
\newblock Doctoral Thesis.

\bibitem{KS}
D.~Kiani and S.~Saaedi.
\newblock Binomial edge ideals of graphs.
\newblock {\em Electronic Journal of Combinatorics}, 19:P44, 2012.

\bibitem{AK}
A.~Kumar.
\newblock Regularity bound of generalized binomial edge ideal of graphs.
\newblock {\em Journal of Algebra}, 546:357--369, 2020.

\bibitem{LMRR}
A.~Lerda, C.~Mascia, G.~Rinaldo, and F.~Romeo.
\newblock $({S}_2)$-condition and {C}ohen-{M}acaulay binomial edge ideals.
\newblock {\em Journal of Algebraic Combinatorics}, 57:589--615, 2022.

\bibitem{MO}
M.~Ohtani.
\newblock Graphs and ideals generated by some 2-minors.
\newblock {\em Communications in Algebra}, 39:905--917, 2011.

\bibitem{JR}
J.~Rauh.
\newblock Generalized binomial edge ideals.
\newblock {\em Adv. Appl. Math.}, 50(3):409–--414, 2013.

\bibitem{R2}
G.~Rinaldo.
\newblock {C}ohen–{M}acaulay binomial edge ideals of cactus graphs.
\newblock {\em Journal of Algebra and Its Applications}, 18:1--18, 2019.

\bibitem{MK2}
S.~Saeedi~Madani and D.~Kiani.
\newblock On the binomial edge ideal of a pair of graphs.
\newblock {\em The Electronic Journal of Combinatorics}, 20:P48, 2013.

\bibitem{MTW}
K.~N. Vander~Meulen, A.~Van~Tuyl, and C.~Watt.
\newblock {C}ohen–{M}acaulay circulant graphs.
\newblock {\em Communications in Algebra}, 42:1896--1910, 2014.

\bibitem{RV}
R.~Villarreal.
\newblock {\em Monomial Algebras}.
\newblock Chapman \& Hall/CRC Monographs and Research Notes in Mathematics. CRC
  Press, 2nd edition, 2018.

\end{thebibliography}

\end{document}